\documentclass[reqno,12pt,letterpaper]{amsart}
\usepackage[proof]{jwmacros}

\usepackage{enumerate}
\usepackage{graphicx}				
\usepackage{amssymb}
\usepackage{amsmath}
\usepackage{amsthm,amsfonts,bm,bbm}
\usepackage{mathtools}
\usepackage{mathrsfs}
\usepackage{enumerate,enumitem}
\usepackage{leftidx}
\numberwithin{equation}{section}
\usepackage{marginnote}
\usepackage{subcaption}

\definecolor{roColor}{RGB}{183, 40, 46}
\definecolor{roColorBlue}{RGB}{40, 148, 200}
\definecolor{roColorGreen}{RGB}{0, 123, 67}

\newcommand{\edra}{\begin{color}{roColorBlue}}
\newcommand{\edrb}{\end{color}}

\title[H\"older damping]{H\"older damping for fractional wave equations}

\author{Jian Wang}
\address{Institut des Hautes {\'E}tudes Scientifiques, 91893 Bures-sur-Yvette, France}
\email[Jian Wang]{wangjian@ihes.fr}

\author{Ruoyu P. T. Wang}
\address{Department of Mathematics\\Yale University\\New Haven\\CT 06511\\United States
}
\email[Ruoyu~P.~T.~Wang]{ruoyu.wang@yale.edu}

\date{}

\renewcommand{\supp}{\operatorname{supp}}

\renewcommand{\Op}{\text{Op}}

\renewcommand{\Re}{\text{Re }}
\renewcommand{\Im}{\text{Im }}

\newcommand{\cim}{\operatorname{Im}}
\newcommand{\cre}{\operatorname{Re}}

\newcommand{\RR}{\mathbb{R}}

\newcommand{\md}{\mathrm{d}}

\begin{document}

\begin{abstract}
For fractional wave equations with low H\"older regularity damping, we establish quantitative energy decay rates for their solutions when the geometric control condition holds. The energy decay rates depend explicitly on the H\"older regularity of the damping. In particular, we show damping functions with lower H\"older regularities that below a certain threshold give slower energy decay.
\end{abstract}

\maketitle

\section{Introduction}

Let $(M,g)$ be an $n$-dimensional closed Riemannian manifold. We study the damped fractional wave equation with initial data $(u_0, u_1)\in H^{\alpha}(M)\times H^{\frac{\alpha}{2}}(M)$
\begin{equation}\label{equation:fractional_wave} 
(\partial_t^2 + \chi\partial_t + |D|^\alpha)u(t,x)=0, \ u(0,x)=u_0, \ \partial_t u(0,x)=u_1 
\end{equation}
where $\alpha\in (0,2)$, $|D|^\alpha\coloneq (-\Delta_g)^{\alpha/2}$ and $\Delta_g$ is the Laplace--Beltrami operator on $M$. The function $\chi$ is non-negative on $M$ and is called the damping function.
For a solution $u$ to \eqref{equation:fractional_wave}, we define its energy at time $t>0$ as
\begin{equation*}
    E(u,t)\coloneqq \int_M \left(||D|^{\alpha/2}u|^2+|\partial_t u|^2 \right) \md \mathrm{vol}_g.
\end{equation*}
Our main result provides explicit energy decay rates as time goes to infinity.

\noindent
{\bf Theorem.}
{\em Suppose $\sqrt{\chi}\in C^{0,\beta}(M; \RR_{\geq 0})$ with $\beta\in [0,1]$ and satisfies the geometric control condition as in Definition \ref{def:gcc}. Then there exists $C>0$ such that, uniformly for any $(u_0,u_1)\in H^\alpha(M)\times H^{\frac{\alpha}{2}}(M)$, for all $t>0$
\begin{equation*}
E(u,t)\le C\langle t\rangle^{-\gamma_\#}(\|u_0\|_{H^{\alpha}}^2+\|u_1\|_{H^{\frac{\alpha}{2}}}^2)
\end{equation*}
where the decay rate $\gamma_\#$ is given by 
\[ \gamma_\#\coloneqq \frac{2}{1-2(1+\frac{\nu_\#}{\alpha})}, \ \nu_\#\coloneqq \min\left(-1, 2\beta+\frac{\alpha}{2}-2\right). \]
}

Here $C^{0,\beta}(M)$ consists of functions that have $\beta$-H\"older regularity. More precisely, a function $f\in C^{0,\beta}(M)$ if the following norm is finite in any local coordinate charts:
\[ \|f\|_{C^{0,\beta}}\coloneqq \sup_{x} |f(x)| + \sup_{x\neq y}\frac{|f(x)-f(y)|}{|x-y|^{\beta}}. \]
Notice that if $\beta>0$ then functions in $C^{0,\beta}$ are continuous, while $C^{0,0}(M)$ is just the space of bounded functions on $M$, that is, $L^{\infty}(M)$.
For damping functions in $C^{0,\beta}(M)$, we define the {\em geometric control conditions} in the sense of \cite{rt74,blr92,leb93}. 
\begin{defi}\label{def:gcc}
    We say $\sqrt \chi\in C^{0,\beta}(M;\RR_{\geq 0})$, $\beta\in [0,1]$, satisfies geometric control condition if there exists an open set $\mathcal O, \tilde{\mathcal O} \subset M$ and $T>0$ such that $\overline{\mathcal O}\subset \tilde{\mathcal O}$, $\inf_{\tilde{\mathcal O}}\chi>0$ and for all $x\in M$, any geodesic on $M$ with unit speed starting at $x$ intersects $\mathcal O$ within time $T$.
\end{defi}

Damped fractional wave equation with $\alpha=1$ appears as the leading order equation for linearized gravity water waves. See works by Clamond et al \cite{clamondetal2005}, Moon \cite{moon2022}, and Alazard, Marzuola and Wang \cite{amw23} for the full damped gravity water wave model and its linearization.
A direct computation shows that 
\begin{equation*}
    \frac{\md}{\md t}E(u,t) = -\int_M \chi |\partial_t u|^2 \, \md \mathrm{vol}_g \leq 0,
\end{equation*}
that is, the energy of the solution decays in time. It is then curious to ask that quantitatively how fast the energy decays to zero for different types of damping functions. For $\alpha=1$, in \cite{kw23}, Kleinhenz and Wang showed that if $\chi\in L^\infty(M)$ and $\chi$ is bounded below by a positive constant on some open set, then the energy decays logarithmically, that is, $E(u,t)\leq C/\log(2+t)$ for some $C>0$ and all $t>0$; if $\chi\in L^{\infty}$ and $\supp \chi$ satisfies the geometric control condition, then the energy decays at the rate $\langle t\rangle^{-1}$, that is, $E(u,t)\leq C\langle t\rangle^{-1}$ for some $C>0$ and all $t>0$. We remark that the later case is exactly the case $\alpha=1$, $\beta=0$ in our Theorem. For more regular damping, in \cite{amw23}, the authors proved that for $\alpha=1$, $M=\mathbb R/2\pi \mathbb Z$ and $\chi\in C^{0,2\beta}$ with $\beta>\frac14$, the energy is bounded by $C\langle t\rangle^{-2}$ for some $C>0$ and all $t>0$. In particular, the $\langle t\rangle^{-2}$ decay rate is optimal for $C^{\infty}$ damping functions. If $\chi$ has high H\"older regularity and has only finite degeneracy in the sense of \cite[Definition 1.1]{amw23}, then the decay rate can by improved, see \cite[Theorem 1]{amw23} for the precise statement.

\captionsetup[subfigure]{labelformat=parens, labelsep=space}
\renewcommand{\thesubfigure}{\Alph{subfigure}}

\begin{figure}[t]
   \centering
   \begin{subfigure}{0.49\textwidth}
       \centering
       \includegraphics[width=\textwidth]{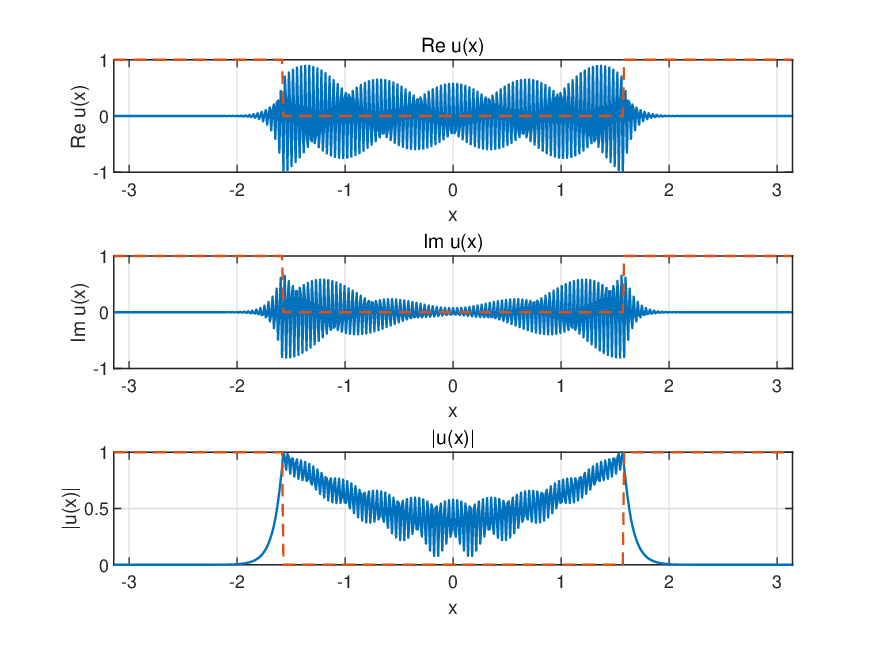}
       \caption{$\chi\in L^{\infty}(\mathbb T)$}
       \label{fig:beta0}
   \end{subfigure}
      \begin{subfigure}{0.49\textwidth}
       \centering
       \includegraphics[width=\textwidth]{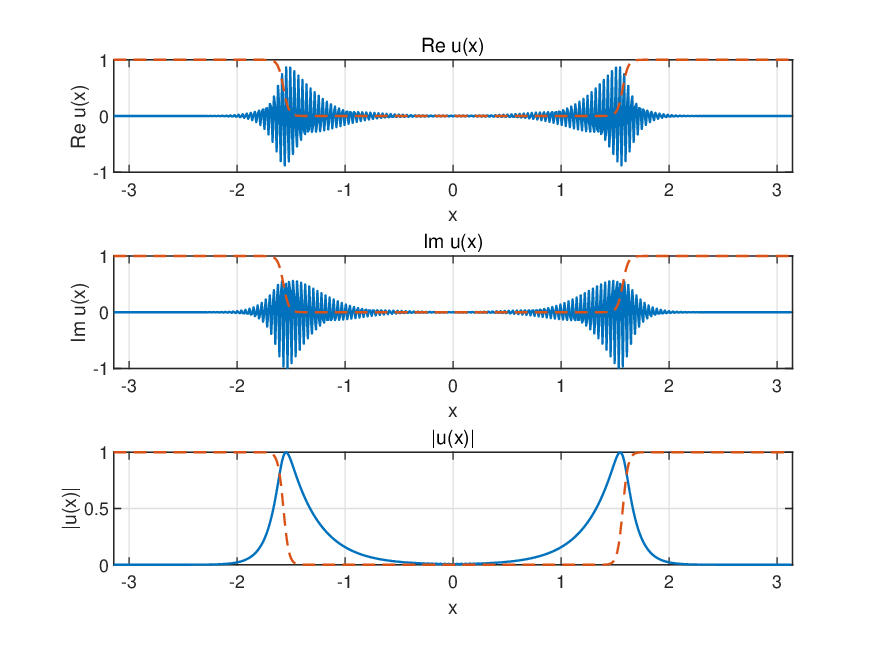}
       \caption{$\chi\in C^{\infty}(\mathbb T)$}
       \label{fig:tanh}
   \end{subfigure}
   \caption{Numerical illustration of eigenfunctions for $P(\lambda)$ in \S \ref{sec:proof} for damping with different regularities on the circle $\mathbb T=\RR/2\pi \mathbb Z$. Blue curves are the real parts, imaginary parts, and absolute values of the eigenfunctions. Red dashed curves are the damping functions. (A)~Damping $\chi = \mathbbm 1_{(-\frac{\pi}{2},\frac{\pi}{2})}$. Eigenvalue $\lambda\sim 13.03-0.03i$. (B) Damping $\chi(x) = 1+\frac12(\tanh(20(x-\frac{\pi}{2}))-\tanh(20(x+\frac{\pi}{2})))$. Eigenvalue $\lambda\sim 13.01-0.14i$. }
    \label{fig:mode}
\end{figure}

In the present paper, we explore energy decay rates for damping functions with low H\"older regularity, in particular, for $\alpha=1$ and $\sqrt{\chi}\in C^{0,\beta}$ with $\beta\in [0,\frac14]$. Unlike in the high H\"older regularity case (for example $\beta>\frac14$ with $\alpha=1$), surprisingly, our main result shows that the energy decay rates depend on the H\"older regularity $\beta$. Moreover, when $\beta$ decreases, the energy decay rate $\gamma_\#$ also decreases. Heuristically, this is because damping functions with lower H\"older regularity acts more like a ``hard wall'' near boundaries of their supports, making it harder for waves to propagate into the damping region. Such intuition is numerically illustrated in Figure~\ref{fig:mode}. In Figure~\ref{fig:beta0} we choose $\chi=\mathbbm 1_{(-\frac{\pi}{2},\frac{\pi}{2})}$ and the boundary of support of $\chi$ is close to a hard wall. In this case, waves are ``trapped'' in the non-damped region, resulting in slower energy decay. In Figure \ref{fig:tanh}, we choose $\chi$ to be a smooth function. The transition from non-damped region to damped region is much ``softer'' now. This makes it easier for waves to travel into the damped region and get damped. As a result, the energy decays faster in this case. For H\"older damping, this heuristic on wave propagation is quantitatively justified by Lemma \ref{lemma:propagate}.

An interesting question is whether energy decay rates in main Theorem are optimal --- in fact, we conjecture they are. According to Anantharaman and L\'eautaud \cite[Proposition 2.4]{aln14} (see also \cite[Proposition 5.1]{amw23}), on $\mathbb T\coloneqq \RR/2\pi \mathbb Z$, this amounts to show the following statement: 
for $\beta\in [0,1]$, let $\chi_\beta(x)=(\cos{x})^{2\beta}\mathbbm 1_{(-\frac{\pi}{2},\frac{\pi}{2})}(x)$ where $\mathbbm 1_{(-\frac{\pi}{2},\frac{\pi}{2})}$ is the indicator function of $(-\frac{\pi}{2},\frac{\pi}{2})$. Then there exists a sequence of real numbers $\lambda_n$ and a sequence of functions $u_n$ such that 
\[ \lambda_n\to \infty, \  \|u_n\|_{L^2}=1, \ \|(|D|^\alpha - i\lambda_n \chi_\beta-\lambda^2_n) u_n\|_{L^2} \leq C\lambda_n^{2(1+\frac{\nu_\#}{\alpha})} \]
where $\nu_\#$ is as in Theorem. 

\vspace{8pt}
\noindent
{\bf Structure of the paper.}
In \S \ref{sec:semi}, we briefly review the calculus of semiclassical pseudodifferential operators and prove a key ingredient of this paper, namely the commutator estimates for semiclassical pseudodifferential operator and H\"older functions, Lemma \ref{lem:commutator}. In \S \ref{sec:proof}, we prove the main Theorem by establishing a semiclassical resolvent estimate for the stationary operator.

\vspace{8pt}
\noindent
{\bf Acknowledgment.}
The authors are grateful to Jeff Galkowski for pointing out a way to sharpen the commutator estimate. JW is supported by Simons Foundation through a postdoctoral position at Institut des Hautes \'Etudes Scientifiques. RPTW was partially supported by EPSRC grant EP/V001760/1. 

\section{Commutators with H\"older functions}\label{sec:semi}

In \S \ref{ssect:semi}, we recall necessary facts about semiclassical analysis. For more details, see \cite{zwo12} or \cite[Appendix E]{dz19}. In \S \ref{ssec:comm}, we quantitatively estimate commutators between semiclassical pseudodifferential operators with compactly support symbols and H\"older functions.
\subsection{Semiclassical calculus}\label{ssect:semi}
Let $T^*M$ be the cotangent bundle of $M$. We say $a$ is in the symbol class $S$ if $a\in C^\infty(T^*M)$ and for all multi-indexes $\varpi,\varrho$,
\[ |\partial_x^\varpi\partial_\xi^\varrho a(x,\xi)|\leq C_{\varpi,\varrho} \text{ for all } (x,\xi)\in T^*M. \]
For a symbol $a\in S$, its semiclassical quantization defines an operator $\Op_h(a)$, also denoted $a(x,hD)$, as 
\[ a(x,hD)u(x)\coloneqq \frac{1}{(2\pi h)^n}\iint e^{\frac{i}{h}\langle x-y,\xi\rangle}a(x,\xi)u(y)\,\md y \md \xi. \]
We say $a(x,hD)$ is a semiclassical pseudodifferential operator. We say $a$ is the semiclassical symbol of $a(x,hD)$ and denote $\sigma_h(a(x,hD))\coloneqq a(x,\xi)$.
If $a, b\in S$, then $a(x,hD)\circ b(x,hD)$ is still a semiclassical pseudodifferential operator with symbol
\[ a\#b = ab+\frac{h}{2i}\{a,b\} + O_S(h^2) \text{ where } \{a,b\}\coloneqq \partial_\xi a\cdot \partial_x b-\partial_x a\cdot \partial_\xi b. \]
Consequently, the commutator $[a(x,hD), b(x,hD)]$ is a semiclassical pseudodifferential operator with symbol 
\[ \sigma_h([a(x,hD), b(x,hD)]) = a\#b-b\# a = -ih\{a,b\} + O_S(h^2). \]
Finally, we recall that pseudodifferential operators $a(x,hD)$ with $a\in S$ are bounded operator on $L^2(M)$ with norms bounded by semi-norms of $a$.

\subsection{Commutator with H\"older functions}\label{ssec:comm}
In the proof the propagation estimate Lemma \ref{lemma:propagate}, we need to estimate the commutator for semiclassical pseudodifferential operators and H\"older functions. In \cite[Proposition 2.1]{amw23}, such an estimate was obtained using paradifferential calculus. Here we sharpen \cite[Proposition 2.1]{amw23} (allowing $\alpha=\beta$ for $\beta\in [0,1]$ in the notations there), by adapting an argument of Galkowski and Wunsch \cite[Lemma 8.1]{gw24}, where the case $\beta=1$ was addressed.
\begin{lemm}\label{lem:commutator}
There exists $C>0$ such that for all $f\in C^{0,\beta}(M)$ with $\beta\in [0,1]$ and $a\in C^{\infty}_c(T^*M)$, for all $h>0$
\[ \|[f, a(x,hD)]\|_{L^2\to L^2}\leq Ch^{\beta}\|f\|_{C^{0,\beta}} . \]
\end{lemm}
\begin{proof}
Let $K(x,y)$ be the integral kernel of the commutator $[f, a(x,hD)]$ in local coordinates, then
\[\begin{split} 
K(x,y)= & \frac{1}{(2\pi h)^n} \int_{\RR^n} e^{\frac{i}{h} \langle x-y,\xi \rangle } (f(x)-f(y)) a(x,\xi) \,\md\xi.
\end{split}\]
For $h>0$ we define 
\[ L_\xi\coloneqq \frac{h+\langle x-y, D_{\xi} \rangle}{ h + \frac{1}{h}|x-y|^2 }. \]
One can check that
\[ L_\xi \left( e^{\frac{i}{h}\langle x-y,\xi \rangle} \right) = e^{\frac{i}{h}\langle x-y,\xi \rangle}, \ L_\xi^T=\frac{h-\langle x-y,D_\xi \rangle}{ h+\frac{1}{h}|x-y|^2 }. \]
Thus we can integrate by parts and write for any $N\in \mathbb N$
\[\begin{split} 
K(x,y) = & \frac{1}{(2\pi h)^n}\int L_\xi^N\left(e^{\tfrac{i}{h}\langle x-y,\xi \rangle}\right)(f(x)-f(y))a(x,\xi)\,\md \xi \\
= & \frac{1}{(2\pi h)^n}\int e^{\tfrac{i}{h}\langle x-y,\xi \rangle}(f(x)-f(y)) (L_\xi^T)^N a(x,\xi)\,\md \xi
\end{split}\]
Since $a\in C_c^{\infty}(T^*M)$,
\[ |(L_\xi^T)^N a(x,\xi)|\leq  \frac{C_N(h+|x-y|)^N}{(h+h^{-1}|x-y|^2)^N} \leq  C_N \langle h^{-1}|x-y| \rangle^{-N}. \]
Now using the H\"older continuity of $f$, we find
\[\begin{split} 
|K(x,y)|\leq & C h^{-n}|f(x)-f(y)|\langle h^{-1}|x-y| \rangle^{-N} \\
\leq & C \|f\|_{C^{0,\beta}} h^{-n+\beta} (h^{-1}|x-y|)^{\beta} \langle h^{-1}|x-y| \rangle^{-N} \\
\leq & C \|f\|_{C^{0,\beta}} h^{-n+\beta} \langle h^{-1}|x-y| \rangle^{-N+\beta}.
\end{split}\]
This pointwise bound implies that by choosing $N>0$ large, we have
\[ \sup_x \int |K(x,y)| \,\md y + \sup_y \int |K(x,y)| \, \md x \leq C h^{\beta}\|f\|_{C^{0,\beta}}\]
which, by Schur test, yields the $L^2$-bound for the commutator.
\end{proof}

\section{Resolvent estimates and Proof of Theorem}\label{sec:proof}
In this section we prove the resolvent bound for the stationary operator
\[ P(\lambda)\coloneqq |D|^{\alpha} - i\lambda \chi - \lambda^2. \]
We introduce the semiclassical scale $\lambda = h^{-\frac{\alpha}{2}}z$ for $z\in \RR$ near $1$ and $h>0$, then we can write
\[ P(\lambda) = h^{-\alpha}P(h,z), \ \text{where} \ P(h,z)\coloneqq |hD|^\alpha -izh^{\frac{\alpha}2}\chi-z^2. \]
Our goal is to show the following semiclassical resolvent estimates.
\begin{prop}[Resolvent estimate for $C^{0,\beta}$-damping]\label{prop:resolvent}
    Suppose $\sqrt{\chi}\in C^{0,\beta}$ with $\beta\in [0,1]$ and $\sqrt \chi$ satisfies the geometric control condition, then there exists $C>0$ and $h_0>0$ such that for all $z\in \RR$ near $1$ and all $0<h<h_0$
\begin{equation*}
    \|P(h,z)^{-1}\|_{L^2\to L^2}\leq Ch^{\nu_\#}
\end{equation*}
    where $\nu_\#= \min(-1, 2\beta+\frac{\alpha}{2}-2)$.
\end{prop}

\subsection{Elliptic regime}
Let $G_0$, $G_{\infty}\in C^{\infty}(T^*M;[0,1])$ such that 
\[ G_0|_{|\xi|\leq \frac18}=1, \ \supp G_0\subset \{|\xi|\leq 1/4\}, \ G_{\infty}|_{|\xi|\geq 8}=1, \ \supp G_{\infty}\subset \{|\xi|\geq 4\}. \] 
Then we have 
\begin{lemm}\label{lemma:low_high}
    For $u\in C^{\infty}$, we have estimates 
    \begin{equation*}\begin{split}
        \|G_0(hD)u\|_{L^2}\leq & \tfrac{1}{4^\alpha|z|^2}\|u\|_{L^2}+\tfrac{1}{|z|^2}\|P(h,z)u\|_{L^2}+ \tfrac{\sqrt{h}}{|z|}\|\sqrt{\chi}\|_{L^\infty}\|\sqrt{\chi}u\|_{L^2}, \\
        \|G_{\infty}(hD)u\|_{L^2}\leq & \tfrac{|z|^2}{4^\alpha}\|u\|_{L^2} + \|P(h,z)u\|_{L^2} + \tfrac{|z|\sqrt{h}}{4^\alpha} \|\sqrt{\chi}\|_{L^{\infty}} \|\sqrt{\chi} u\|_{L^2}.
    \end{split}\end{equation*}
\end{lemm}
\begin{proof}
    For the estimate of $G_0(hD)u$, we use the identity that 
    \[ z^2G_0(hD)u = G_0(hD)|hD|^{\alpha} u - G_0(hD)P(h,z)u - iz\sqrt{h}G_0(hD)\chi u. \]
    By Plancherel theorem and the support condition on $G_0$, 
    \[\begin{split} 
    & \|G_0(hD)|hD|^\alpha u\|_{L^2}\leq  \frac{1}{4^\alpha}\|u\|_{L^2}, \\
    & \|G_0(hD)P(h,z)u\|_{L^2}\leq  \|P(h,z)u\|_{L^2}, \\ 
    & \|G_0(hD)\chi u\|_{L^2}\leq  \|\chi u\|_{L^2}\leq \|\sqrt\chi\|_{L^\infty}\|\sqrt\chi u\|_{L^2} .
    \end{split}\]
    Now the first inequality follows from the triangle inequality.
    
    The second inequality is obtained similarly by using 
    \[ G_{\infty}(hD)|hD|^\alpha u = G_{\infty}(hD)P(h,z)u+iz\sqrt{h}G_{\infty}(hD)\chi u + z^2 G_{\infty}(hD)u \]
    and noticing the support condition on $G_{\infty}$.
\end{proof}

\subsection{Propagation regime}
Let us now study the propagation at intermediate frequencies, that is, $|\xi|\sim 1$. 
Let $G\in C^{\infty}_c(T^*M;[0,1])$ such that $\supp G\subset \{1/16\leq |\xi|\leq 16\}$ and $G=1$ on $\{1/8\leq |\xi|\leq 8\}$.
\begin{lemm}\label{lemma:propagate}
    Suppose $\sqrt{\chi}\in C^{0,\beta}(M)$, $\beta\in [0,1]$, and $\sqrt \chi$ satisfies the geometric control condition, then for all $N\in \RR$, there exist $h_0, C>0$ such that for all $\epsilon>0$, $z\in \RR$ near $1$, $0<h<h_0$, and $u\in C^{\infty}(M)$,
    \[ \|G(hD)u\|_{L^2}\leq Ch^{-1}\|P(h,z)u\|_{L^2}+C (\sqrt h + \epsilon)\|u\|_{L^2}+C(1+\epsilon^{-1}h^{\beta+\frac{\alpha}{2}-1})\|\sqrt{\chi}u\|_{L^2}. \]
\end{lemm}

\Remark
The term $(1+\epsilon^{-1} h^{\beta-1+\frac{\alpha}2})\|\sqrt{\chi}u\|$ quantitatively measures how the H\"older regularity $\beta$ of $\chi$ obstructs wave propagation.

\begin{proof}
{\bf 1.} Since $\sqrt\chi$ satisfies the geometric control condition, we can take open sets $\mathcal O$ and $\tilde{\mathcal O}$ as in Definition \ref{def:gcc}. Then for any $(x_0,\xi_0)\in \supp G$, there exists $T>0$ such that 
\[ \exp(TH_p)(x_0,\xi_0)\in T^*\mathcal O, \ \text{ where } p=|\xi|, \ H_p = |\xi|^{-1}\xi\cdot \partial_x. \]
Consequently, there exists an escape function $f\in C_c^{\infty}(T^*M; \RR)$ such that
\begin{itemize}
    \item $f$ is compactly supported in $T^*M\setminus 0$;
    \item $f\geq 0$ on $T^*M$;
    \item $f>0$ on $\supp G$;
    \item There exists $\gamma\geq 0$ such that $H_pf\leq -\gamma f$ near $T^*(M\setminus \mathcal O)$.
\end{itemize}
For the construction of such an $f$, we refer to \cite[Lemma E.48]{dz19}.
Let $\psi=\psi(\xi)\in C^{\infty}(T^*M)$ such that $\psi$ is compactly supported in $T^*M\setminus 0$ and $\psi=1$ on $\supp f$.
Let $F\coloneqq f(x,hD)\psi(hD)$ where $f(x,hD)\coloneqq \Op_h(f)$ is the semiclassical quantization of $f$. Then $F$ is a semiclassical pseudodifferential operator with semiclassical symbol $\sigma_h(F) = f\psi$. Let us compute the commutator
\begin{equation}\label{eq:comm} 
\cim\langle P(h,z), F^*F u\rangle=\cim\langle (\cre P(h,z)) u, F^*F u\rangle+\cre\langle (\cim P(h,z)) u, F^*F u\rangle. 
\end{equation}
Here $\Re P(h,z)$ and $\Im P(h,z)$ are defined as 
\[\begin{split}
& \Re P(h,z) \coloneqq \frac{P(h,z)+P(h,z)^*}{2} = |hD|^\alpha -z^2, \\
& \Im P(h,z) \coloneqq \frac{P(h,z)-P(h,z)^*}{2} = -z h^{\frac{\alpha}{2}} \chi.
\end{split}\]
We estimate the two terms on the right-hand-side of \eqref{eq:comm} separately.

{\bf 2.} We first estimate the first term on the right-hand-side of \eqref{eq:comm}. We compute
\begin{equation*}\begin{split}
\cim\langle (\cre P(h,z)) u, F^*F u\rangle = \left\langle \tfrac{1}{2i}[F^*F, |hD|^\alpha]u, u \right\rangle.
\end{split}\end{equation*}
Notice that if $\tilde \psi = \tilde \psi(\xi)\in C^{\infty}(T^*M)$ such that $\tilde \psi$ is compactly supported in $T^*M\setminus 0$ and $\tilde \psi =1$ on $\supp \psi$, then one can check
\[\begin{split} 
[F^*F, |hD|^\alpha] = & F^*f(x,hD)\psi(hD)|hD|^\alpha - |hD|^\alpha \psi(hD)f(x,hD)^* F \\
= & F^*f(x,hD)\psi(hD) \left(\tilde \psi(hD)|hD|^\alpha\right) - \left(\tilde \psi(hD)|hD|^\alpha\right) \psi(hD)f(x,hD)^* F \\
= & [F^*F, \tilde \psi(hD)|hD|^\alpha]. 
\end{split}\]
Since $\tilde\psi$ is compactly supported in $T^*M\setminus 0$, we know $|\xi|^\alpha \tilde \psi\in C^{\infty}_c(T^*M)$. As a result, we find $\tilde \psi(hD)|hD|^{\alpha}$ is a semiclassical pseudodifferential operator and has semiclassical symbol $|\xi|^\alpha\tilde \psi(\xi)$. Thus we have 
\[ \sigma_h\left( \tfrac{1}{2i h}[F^*F, |hD|^\alpha] \right) = \alpha |\xi|^{\alpha-1}\psi(\xi)^2 f H_{p}f. \]
In particular, by the assumptions on $f$ and $\psi$, we find for $h$ small enough 
\[  \sigma_h\left( \tfrac{1}{2i h}[F^*F, |hD|^\alpha] \right) \leq -\alpha\gamma |\xi|^{\alpha-1}\psi^2 f^2\leq -cf^2 \text{ near } T^*(M\setminus \mathcal O) \]
for some $c>0$. 
Let $\chi_0\in C_c^\infty(\tilde{\mathcal O};[0,1])$ be such that $\chi_0=1$ on $\mathcal O$. Then there exists $C>0$ such that
\[ \sigma_h\left( -\tfrac{1}{2i h}[F^*F, |hD|^\alpha]-cF^*F + C\chi_0^2 \right)\geq 0 \text{ on } T^*M. \]
Apply microlocal G{\aa}rding's inequality (see \cite[Theorem 4.32]{zwo12} or \cite[Proposition E.34]{dz19}) and we claim that for $h>0$ sufficiently small and $u\in C^{\infty}(M)$
\[ \left\langle \left( -\tfrac{1}{2i h}[F^*F, |hD|^\alpha]-cF^*F + C\chi_0^2 \right)u, u \right\rangle \geq -Ch\|u\|_{L^2}^2. \]
Rearranging the terms, noticing $\|\chi_0 u\|_{L^2}\leq C\|\sqrt \chi u\|_{L^2}$, we obtain 
\begin{equation}\begin{split}\label{eq:real} 
h^{-1}\Im \langle (\Re P(h,z))u,F^*Fu \rangle 
\leq & -c\|Fu\|_{L^2}^2+ C \|\sqrt \chi u\|_{L^2}^2 + Ch\|u\|_{L^2}^2.
\end{split}\end{equation}

{\bf 3.} Now we consider the second term on the right-hand-side of \eqref{eq:comm}. Note that
\begin{equation*}
\cre\langle (\cim P(h,z)) u, F^*F u\rangle=-z h^{\frac{\alpha}{2}}\cre\langle \chi u, F^*F u\rangle. 
\end{equation*}
We compute
\begin{equation*}
\begin{split}
\cre{\langle \chi u, F^*F u\rangle}
= &\|F \sqrt{\chi}u\|^2+\cre{\langle\sqrt{\chi}u, [\sqrt\chi,F^*F]u\rangle} \\
\ge & \cre{\langle\sqrt{\chi} u, [\sqrt{\chi},F^*F]u\rangle}. 
\end{split}
\end{equation*}
By Lemma \ref{lem:commutator} we have
\begin{equation*}
\|[\sqrt{\chi},F^*F]u\|_{L^2} \le C h^{\beta}\|u\|_{L^2}. 
\end{equation*}
Thus we have
\begin{equation*}\begin{split}
\cre{\langle\sqrt{\chi}u, [\sqrt{\chi},F^*F]u\rangle}
\geq & -\|\sqrt \chi u\|_{L^2}\|[\sqrt \chi, F^*F]u\|_{L^2} \\
\geq &  -C h^{\beta} \|\sqrt \chi u\|_{L^2}\|u\|_{L^2} \\
\ge & -\epsilon h^{1-\frac{\alpha}{2}}\|u\|_{L^2}^2-C\epsilon^{-1} h^{2\beta-1+\frac{\alpha}{2}}\|\sqrt{\chi}u\|_{L^2}^2
\end{split}\end{equation*}
for any $\epsilon>0$. Therefore
\begin{equation}\label{eq:imag}
h^{-1}\cre\langle (\cim P(h,z)) u, F^*F u\rangle\le C\epsilon \|u\|_{L^2}^2+C\epsilon^{-1} h^{2\beta-2+\alpha}\|\sqrt{\chi}u\|_{L^2}^2. 
\end{equation}

{\bf 4.} Combining \eqref{eq:real} and \eqref{eq:imag} and we conclude
\[ h^{-1}\Im\langle P(h,z)u,F^*Fu \rangle \leq -c\|Fu\|_{L^2}^2 + C(1+\epsilon^{-1}h^{2\beta+\alpha-2})\|\sqrt\chi u\|_{L^2}^2 + C(h+\epsilon)\|u\|_{L^2}^2.  \]
Therefore
\[ \|Fu\|_{L^2}^2 \leq Ch^{-1}\|FP(h,z)u\|_{L^2}\|Fu\|_{L^2} + C(1+\epsilon^{-1}h^{2\beta+\alpha-2})\|\sqrt\chi u\|_{L^2}^2 + C(h+\epsilon)\|u\|_{L^2}^2.  \]
Using Cauchy--Schwartz and $F\in C^{\infty}_c(T^*M\setminus 0)$ and we obtain 
\[ \|Fu\|_{L^2} \leq Ch^{-1}\|P(h,z)u\|_{L^2} + C(1+\epsilon_1^{-1}h^{\beta+\frac{\alpha}{2}-1})\|\sqrt\chi u\|_{L^2} + C(\sqrt h+\epsilon_1)\|u\|_{L^2}  \]
where $\epsilon_1\coloneqq \sqrt \epsilon$ is an arbitrary positive number. The proof is completed upon redefining $\epsilon$ as $\epsilon_1$ and noticing that by the elliptic estimates \cite[Theorem E.33]{dz19}
\[ \|G(hD)u\|_{L^2}\leq \|Fu\|_{L^2}+O(h^{\infty})\|u\|_{L^2} \]
since $F>0$ on $\supp G$.
\end{proof}

\begin{proof}[Proof of Proposition \ref{prop:resolvent}]
    Indeed, summing the estimates in Lemma \ref{lemma:low_high} and \ref{lemma:propagate}, we obtain for $0<h\ll 1$,
    \begin{equation}\label{eq:dc}
        \|u\|_{L^2}\leq Ch^{-1}\|P(h,z)u\|_{L^2} + C(1 + h^{\beta-1+\frac{\alpha}2})\|\sqrt{\chi}u\|_{L^2}.
    \end{equation}
Next we notice 
    \begin{equation}\begin{split}\label{eq:ec}
        & \Im\langle P(h,z)u,u \rangle = z h^{\frac{\alpha}2}\|\sqrt{\chi}u\|_{L^2}^2 \\
        &  \ \ \Rightarrow \|\sqrt\chi u\|_{L^2}^2 \leq Ch^{-\frac{\alpha}2} \|P(h,z)u\|_{L^2}\|u\|_{L^2} \\
        &  \ \ \Rightarrow \|\sqrt{\chi}u\|_{L^2}\leq C\delta(h)^{-1}h^{-\frac{\alpha}2}\|P(h,z)u\|_{L^2} + \delta(h)\|u\|_{L^2}.
    \end{split}\end{equation}
Combining estimates \eqref{eq:dc} and \eqref{eq:ec}, we find 
    \begin{equation*}\begin{split}
        \|u\|_{L^2}
        \leq & (h^{-1}+ C\delta(h)^{-1}h^{-\frac{\alpha}2}(1+h^{\beta-1+\frac{\alpha}2}) )\|P(h,z)u\|_{L^2} + (1+h^{\beta-1+\frac{\alpha}2})\delta(h)\|u\|_{L^2}.
    \end{split}\end{equation*}
    Take $\delta(h)$ so that $(1+h^{\beta-1+\frac{\alpha}{2}})\delta(h)=\frac12$ and we obtain
    \begin{equation*}\begin{split}
        \|u\|_{L^2}
        \leq & C( h^{-1} + h^{-\frac{\alpha}2}(1+h^{\beta-1+\frac{\alpha}2})^2 )\|P(h,z)u\|_{L^2} \\
        \le & Ch^{\min( -1, -\frac{\alpha}{2}, 2\beta+\frac{\alpha}{2}-2 )}\|P(h,z)u\|_{L^2}.
    \end{split}\end{equation*}
  It remains to notice $\frac{\alpha}{2}<1$ since $\alpha\in (0,2)$ to conclude the proof of Proposition~\ref{prop:resolvent}.
\end{proof}

We are now ready to prove the main Theorem.
\begin{proof}[Proof of Theorem]
    Recall the definition of $P(h,z)$ and we can restate the resolvent estimates in $\lambda$: there exists $C>0$ such that for all $\lambda\in \RR\setminus [-C,C]$, 
    \[ \|P(\lambda)^{-1}\|_{L^2\to L^2}\leq C\lambda^{-2(1+\frac{\nu_\#}{\alpha})}. \]
    The polynomial decay then follows from standard semigroup tools as in \cite[Proposition 2.4]{aln14} (see also \cite{bt10}, \cite[Proposition 5.1]{amw23}, \cite[Lemma~2.25]{kw22}).
\end{proof}

\bibliographystyle{alpha}
\bibliography{Robib}

\end{document}